\documentclass[12pt]{amsart}

\usepackage{amsmath, amsthm, amssymb}
\usepackage{geometry}
\usepackage{hyperref}
\usepackage{enumerate}
\usepackage{cite}
\usepackage{ytableau}
\usepackage{adjustbox}
\usepackage{tikz}
\usepackage{float}
\usepackage{bm}
\usepackage{longtable}
\usepackage{enumitem}
\usepackage{mathtools}

\usetikzlibrary{decorations.pathreplacing}

\numberwithin{equation}{section} 
\newtheorem{theorem}{Theorem}[section]
\newtheorem*{maintheorem}{Main Theorem}
\newtheorem{lemma}[theorem]{Lemma}
\newtheorem{proposition}[theorem]{Proposition}

\theoremstyle{definition}
\newtheorem{definition}[theorem]{Definition}
\theoremstyle{remark}

\newtheorem{example}[theorem]{Example}

\newcommand{\shiftnoarrow}[2]{\ensuremath \raisebox{#1cm}{${#2}$}}
\newcommand{\YoungWithData}[2]{%
	\begin{tabular}{@{}c@{}}
		#1 \\[-0.6em]
		{\small $#2$}
	\end{tabular}
}

\title{Verma Bases and Kashiwara-Nakashima Tableaux of $\mathfrak{sp}_4$}

\author{Bintao Cao$^1$}
\address{1. School of Mathematics and Statistics, Yunnan University, Kunming 650500, China}
\email{btcao@ynu.edu.cn}

\author{Ye Huang$^2$}
\address{2. School of Mathematics and Statistics, Yunnan University, Kunming 650500, China}
\email{huangye2461@163.com}

\date{\today}

\begin{document}
	
	\begin{abstract}
			We construct a one-to-one correspondence between the Verma basis vectors of a finite dimensional irreducible representation $L(\lambda)$ of the symplectic Lie algebra $\mathfrak{sp}_4$ and the Kashiwara-Nakashima tableaux of $\mathfrak{sp}_4$ with shape $\lambda $ naturally.
		We also give a proof of the linear independence of the Verma vector system directly.
	\end{abstract}
	
	\maketitle
	
	\section{Introduction}
	Finding some ``good'' bases is 
	an important problem in representation theory. 
	Various good bases have already been found,
	and have been used to deal with all kinds of problems in the representation theory of Lie groups, Lie algebras, quantum groups and Hecke algebras.
	For example, the well known bases,
	called canonical and dual canonical bases,
	introduced by Kazhdan and Lusztig\cite{KazhdanLusztig1979} can give the characteristic formulas for some important modules in
	BGG category $\mathcal O,$ such as irreducible modules,
	projective modules and tilting modules. 
	The crystal bases for the integrable
	modules of the quantized enveloping algebras of some
	symmetrizable Kac-Moody Lie algebras 
	contain lots of information of these modules at $q=0.$ 
	They are well-behaved 
	with respect to tensor products, 
	and thus can be used to decompose tensor modules.
	They are constructed by Lusztig~\cite{lusztig1990canonical} and Kashiwara~\cite{kashiwara1990crystalizing,kashiwara1991}. 
	We know that the crystal bases 
	also have important connection with combinatorial bases of the tableaux.
	Gelfand-Tsetlin bases, given by Gelfand and Tsetlin~\cite{gelfand1, gelfand2},
	are used to study the branching theory.
	
	In many scenes, monomial bases are useful and beautiful.
	The Verma basis is a monomial basis of certain highest weight modules expressed in terms of the negative simple root vectors, 
	which depends on the reduced expression of the longest element of the Weyl group. 
	We introduce this basis of finite dimensional irreducible module here briefly. 
	
	Let $\mathfrak{g}$ be a finite dimensional simple Lie algebra over $\mathbb C.$
	We write $L(\lambda)$ the irreducible module with the highest weight $\lambda,$
	and $v_{\lambda}$ a highest weight vector with weight $\lambda.$
	Let $w_0=r_{i_1}r_{i_2}\cdots r_{i_t}$ be the longest element in the Weyl group of $\mathfrak{g}$,
	where $r_{i_j}'s$ are simple reflections,
	and this expression has the shortest length.
	Then the Verma basis of $L(\lambda)$ has the form 
	$$f^{a_{i_1}}_{i_1}f^{a_{i_2}}_{i_2}\cdots f^{a_{i_t}}_{i_t}v_{\lambda}.$$
	Here, the indices $a_{i_j}$ satisfy some inequalities.
	These inequalities can be explained as just the conditions for  tableaux to be certain types of tableaux. 
	For example, in type $A,$
	they are the conditions for the tableaux to
	be standard Young tableaux.

	Verma basis first appeared in the work of Li, Moody, Nicolescu, and Patera\cite{li1986verma} in 1986 (due to D-N. Verma).
	They worked out the Verma bases for finite dimensional irreducible representations of the Lie algebras \( A_n (n \geq 1), B_n (2 \leq n \leq 6), C_n (2 \leq n \leq 6), D_n (4 \leq n \leq 6), \) and \( G_2 \). But the proof of the linear independence of these basis vectors is not so clear.
	Raghavan and Sankaran~\cite{Raghavan1999} gave a proof for  \( A_n \). 
	Po{\v{s}}ta and Havl{\'\i}{\v{c}}ek~\cite{posta2013note} discussed the construction of the Verma bases of the Verma modules of  \( A_n \). 
	Hall~\cite{hall1987verma} presented the Verma bases for Verma modules of $A_n$ and some other types of Lie algebras in his Thesis. 
	For the application, the  Verma basis has been used in Quantum Chemistry
	by Paldus and Planelles~\cite{Paldus2018}.
	
	In this paper, we consider Verma bases of finite dimensional irreducible modules of the finite dimensional simple Lie algebra $\mathfrak{sp}_4$ over $\mathbb C$. 
	We give a natural one-to-one correspondence between the set of Verma vectors and the set of Kashiwara-Nakashima tableaux
	in this case,
	via the transformation of the inequalities in the expression of Verma basis to the conditions for tableaux to
	be Kashiwara-Nakashima tableaux naturally. 
	Then we give a direct proof of the linear
	independence of the Verma vector system.
	Exactly, we have the following result.
	
	Let $L(\lambda)$ be the finite dimensional irreducible representation of $\mathfrak{sp}_4$ with highest weight $\lambda$. We write $\lambda = m_1\omega_1 + m_2\omega_2 = (m_1+m_2)\epsilon_1 + m_2\epsilon_2$ for $m_1, m_2 \in \mathbb{Z}_{\geq 0}$, where $\omega_1, \omega_2$ are the fundamental weights. 
	We also write 
	$\lambda=(\lambda_1,\lambda_2)=(m_1+m_2,m_2)$
	be a partition for convenience.
	Let the nonzero vector $v_{\lambda}$ be the highest weight vector of $L(\lambda)$. 
	Then the Verma vector system of $L(\lambda)$ 
	is defined by
	\begin{equation} \label{def:F}
	F:= \left\{ f_1^{a_4} f_2^{a_3} f_1^{a_2} f_2^{a_1} v_{\lambda} \;\middle|\;
	\begin{aligned}
		a_i &\in \mathbb Z_{\ge0}, \;i=1,2,3,4,\\
		0 &\leq a_1 \leq m_2, \\
		0 &\leq a_2 \leq m_1 + 2a_1, \\
		0 &\leq a_3 \leq \min\left\{ \tfrac{1}{2}(a_2 + m_1), a_2 \right\}, \\
		0 &\leq a_4 \leq \min\{ m_1, a_3 \}
	\end{aligned}
	\right\}, 
	\end{equation}
		where $f_1 = E_{21} - E_{34}$ is the root vector corresponding to the negative simple root $\epsilon_2 - \epsilon_1$ of $\mathfrak{sp}_4$, and $f_2 = E_{42}$ is the root vector corresponding to the negative simple root $-2\epsilon_2$. 
	
	\begin{maintheorem}
		Let $L(\lambda)$ be the finite dimensional irreducible representation of $\mathfrak{sp}_4$ with highest weight $\lambda$. Let $F$ be the Verma vector system of  $L(\lambda)$ (see \eqref{def:F}).
		Then 
		\begin{itemize}
		\item[i).] There exists a weight-preserving one-to-one correspondence between $F$ and the set of Kashiwara-Nakashima tableaux of $\mathfrak{sp}_4$ with the shape $\lambda$.
		\item[ii).] $F$ is a basis of $L(\lambda).$ 
		
	\end{itemize}
	\end{maintheorem}
	
	By the way, the connection between the Verma vectors and the Kashiwara-Nakashima tableaux inspires us to give  Verma bases of irreducible modules of some Basic classical Lie superalgebras.
	In fact, 
	we have already given the Verma basis in the case of
	$\mathfrak{spo}(4|1)$, which will appear in our next paper.
	
	In Section 2, we introduce the structure of the symplectic Lie algebra $\mathfrak{sp}_{2n}$ and its Kashiwara-Nakashima tableaux.
	Then in Section 3, we give the basic properties of the
	Kashiwara-Nakashima tableaux of $\mathfrak{sp}_{4},$
	especially, the system of inequalities that makes a tableau a Kashiwara-Nakashima tableau.
	In Section 4, we give the connection of the Verma basis of the finite dimensional irreducible module to 
	the
	Kashiwara-Nakashima tableaux of $\mathfrak{sp}_{4}.$
	In Section 5, we show the linear independence of the Verma vector system.

	\section{Preliminaries}
	In this section, we introduce the symplectic Lie algebra $\mathfrak{sp}_{2n}$ and its Kashiwara-Nakashima tableaux (KN tableaux for short). The KN tableaux is developed by Kashiwara and Nakashima to label elements of symplectic crystal bases \cite{kashiwara1994crystal}.
	
	Let $\mathfrak{g}$ be the finite dimensional simple Lie algebra of type $C_n$ ($n \geq 2$). Then $\mathfrak{g}$ is realized as the symplectic Lie algebra \cite{humphreys1972}.
	
	\begin{equation}
		\mathfrak{g} =\mathfrak{sp}_{2n} = \left\{T = \begin{pmatrix}
			A & B \\
			C & -A^t
		\end{pmatrix}\in M_{2n\times 2n}(\mathbb{C}) \; \middle| \; 
		\begin{array}{l}
			A,B,C\in M_{n\times n}(\mathbb{C}), \\
			B^t = B,\; C^t = C
		\end{array}
		\right\}.
	\end{equation}
	
	Let $E_{ij}$ denote the $2n\times 2n$ elementary matrix having $1$ at the $(i,j)$-entry and $0$ elsewhere, and set 
	\begin{equation}
		\begin{aligned}
			e_i &= E_{i, i+1} - E_{n+i+1, n+i}, \quad f_i = E_{i+1, i} - E_{n+i, n+i+1}, \\
			h_i &= E_{ii} - E_{i+1, i+1} - E_{n+i, n+i} + E_{n+i+1, n+i+1}, \\
			e_n &= E_{n, 2n}, \quad f_n = E_{2n, n}, \quad h_n = E_{n,n} - E_{2n,2n}
		\end{aligned}
	\end{equation}
	for $i = 1, 2, \ldots, n-1$. Then, as a Lie algebra, $\mathfrak{g} = \mathfrak{sp}_{2n}$ is generated by $e_i, f_i, h_i$ ($i = 1,2, \ldots, n$).
	
	Consider the linear functionals $\epsilon_i: M_{2n \times 2n}(\mathbb{C}) \to \mathbb{C}$ defined by
	$$
	\epsilon_i(T) = t_{ii}, \quad T = (t_{rs}) \in M_{2n \times 2n}(\mathbb{C}), \quad r,s = 1,2,\ldots,2n.
	$$
	Then, the simple roots and the fundamental weights are expressed as
	\begin{equation}\label{simple roots}
		\begin{aligned}
			\alpha_i &= \epsilon_i - \epsilon_{i+1} & \text{for } i &= 1,2,\ldots,n-1, \\
			\alpha_n &= 2\epsilon_n, \\
			\omega_i &= \epsilon_1 + \epsilon_2 + \cdots + \epsilon_i & \text{for } i &= 1,2,\ldots,n.
		\end{aligned}
	\end{equation}
	
	The Dynkin diagram of the Lie algebra $\mathfrak{sp}_{2n}$ is shown in the figure below.

\begin{figure}[htbp]
	\centering
	\begin{tikzpicture}
		\tikzstyle{white node}=[circle, draw=black, thick, minimum size=0.4cm, inner sep=0pt, fill=white]
		
		\node[white node, label=below:$\alpha_1$] (n1) at (0, 0) {};
		\node[white node, label=below:$\alpha_2$] (n2) at (1.8, 0) {};
		\node[white node, label=below:$\alpha_3$] (n3) at (3.6, 0) {};
		
		\node (dots) at (5.4, 0) {$\cdots$};
		
		\node[white node, label=below:$\alpha_{n-2}$] (nm2) at (7.2, 0) {};
		\node[white node, label=below:$\alpha_{n-1}$] (nm1) at (9.0, 0) {};
		\node[white node, label=below:$\alpha_n$] (nn) at (10.8, 0) {};
		
		\draw[thick] (n1) -- (n2);
		\draw[thick] (n2) -- (n3);
		\draw[thick] (n3) -- (dots);
		\draw[thick] (dots) -- (nm2);
		\draw[thick] (nm2) -- (nm1);
		
		\draw[thick] ([yshift=0.08cm]nm1.east) -- ([yshift=0.08cm]nn.west);
		\draw[thick] ([yshift=-0.08cm]nm1.east) -- ([yshift=-0.08cm]nn.west);
		
		\draw[thick] (10.15, 0.25) -- (9.8, 0) -- (10.15, -0.25);
		
	\end{tikzpicture}
	\caption{The Dynkin diagram of the Lie algebra $\mathfrak{sp}_{2n}$}
\end{figure}
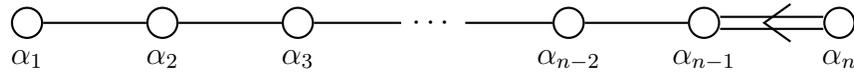

	A partition $\lambda$ of a positive integer $m$ is a $k$-tuple $\lambda=(\lambda_1,\ldots,\lambda_k)$, where $\lambda_1\geq\lambda_2\geq\cdots\geq\lambda_k > 0$ and $|\lambda|=\sum_{i = 1}^k\lambda_i=m$. The length of $\lambda$ is $\ell(\lambda)=k$ and the Young diagram of shape $\lambda$ is given by arranging $m$ boxes in $k$ left-justifed rows with $\lambda_i$ boxes in the $i$th row. Figure~\ref{Young diagram} shows a Young diagram of shape $(4,3,1)$.
	
	A semistandard Young tableau of shape $\lambda$ is a filling of the Young diagram of shape $\lambda$ with positive integers such that the entries are weakly increasing across each row and strictly increasing down each column. Figure~\ref{Young tableau} shows a semistandard Young tableau.
	\begin{figure}[h]
		\centering
		\begin{minipage}{0.45\textwidth}
			\centering
			\begin{ytableau}
				~ & ~ & ~ & ~ \\
				~ & ~ & ~ \\
				~
			\end{ytableau}
			\caption{Young diagram}
			\label{Young diagram}
		\end{minipage}
		\begin{minipage}{0.45\textwidth}
			\centering
			\begin{ytableau}
				1 & 2 & 3 & 3 \\
				3 & 3 & 4 \\
				4
			\end{ytableau}
			\caption{Young tableau}
			\label{Young tableau}
		\end{minipage}
	\end{figure}
	
	Let $\lambda = m_1\omega_1+\cdots + m_n\omega_n=\lambda_1\epsilon_1+\cdots+\lambda_n\epsilon_n$ ($m_i\in\mathbb{Z}_{\geq0}$) be a dominant integral weight. Using the relation \ref{simple roots}, $\lambda$ can be identified with the partition $\lambda = (\lambda_1, \lambda_2, \cdots, \lambda_n)$, where
	\begin{equation}
		\begin{aligned}
			\quad\lambda_1&=m_1 + m_2+\cdots + m_n,\\
			\lambda_2&=m_2+\cdots + m_n,\\
			&\vdots\\
			\lambda_n&=m_n.
		\end{aligned}    
	\end{equation}
	
	Finite dimensional irreducible highest weight representations of $\mathfrak{sp}_{2n}$ are indexed by partitions $\lambda$ with $\ell(\lambda) \leq n$. 
	Therefore, in the definition of KN tableaux of $\mathfrak{sp}_{2n}$,
 the condition $\ell(\lambda) \leq n$ is imposed. 
 KN tableaux of $\mathfrak{sp}_{2n}$ have entries from the set $$\mathcal{N}=\{1,2,\ldots,n,\overline{n},\ldots,\overline{2},\overline{1}\}$$ 
 with ordering
	\[1 < 2<\cdots < n<\overline{n}<\cdots <\overline{2}<\overline{1}.\]
	
	\begin{definition}\label{KN}\cite{kashiwara1994crystal, hong2002introduction, kwon2018}
		Let $\lambda$ be a partition with $\ell(\lambda) \leq n$. A KN tableau $T$ of $\mathfrak{sp}_{2n}$ of shape $\lambda$ is a filling of the Young diagram of shape $\lambda$ with entries from $\mathcal{N}$ that satisfies the following conditions:
		\begin{enumerate}[label=(\roman*)]
			\item The entries in $T$ are weakly increasing across each row and strictly increasing down each column.
			\item For every column in $T$ that contains both an $i$ and an $\overline{i}$, where $i$ belongs to the $p$-th box from the top and $\overline{i}$ belongs to the $q$-th box from the bottom, we have $p + q \leq i$.
			\item If $T$ has two adjacent columns having one of the following configurations, where $p, q, r, s$ are the relevant row numbers (where rows are counted from top to bottom), with $p \leq q < r \leq s$ and $i \leq j$, then $(q - p) + (s - r) < j - i$.
			
			$$\left.\begin{array}{lr}
				p \rightarrow  & i \\
				\\
				q \rightarrow &  \\
				&  \\
				r \rightarrow & \\
				&\\
				s \rightarrow &
			\end{array}
			\right|  \begin{array}{r}  \cr \cr j \cr  \cr \overline{j} \cr  \cr \overline{i} \cr \end{array}\shiftnoarrow{-1.5}{,} \quad \quad\quad\quad
			\left.\begin{array}{r}
				i \\
				\\
				j  \\
				\\
				\overline{j} \\
				\\
				\\
			\end{array}
			\right|  \begin{array}{r}  \cr
				\cr
				\cr
				\cr
				\cr
				\cr
				\overline{i} \cr
			\end{array}\shiftnoarrow{-1.5}{.}$$
		\end{enumerate}
		Let $\mathrm{KN}_{\lambda}(2n)$ denote the set of KN tableaux of $\mathfrak{sp}_{2n}$ of shape $\lambda$.
	\end{definition}
	
	For $T \in \mathrm{KN}_{\lambda}(2n)$, we define its weight to be
	\begin{equation}
		\mathrm{wt}(T) = \sum_{i = 1}^{n} (k_i - k_{\overline{i}}) \epsilon_i,
	\end{equation}
	where $k_i$ (respectively, $k_{\overline{i}}$) is the number of $i$'s (respectively, $\overline{i}$'s) appearing in $T$.

	\begin{example}
		When $n = 3$ and $\lambda=2\omega_3$, the tableau
		$T = \begin{adjustbox}{valign=c}
			\begin{ytableau}
				1 & 3 \\
				2 & \bar{3} \\
				3 & \bar{1}
			\end{ytableau}
		\end{adjustbox}$
		belongs to $\mathrm{KN}_{\lambda}(6)$, because
		\begin{enumerate}[label=(\roman*)]
			\item $p + q =1 +2 =3 \leq 3$,
			\item $(q - p)+(s - r)=(1 -1)+(3 -2)=1 < j - i =3 -1 =2$.
		\end{enumerate}
		The weight of $T$ is $\mathrm{wt}(T)=\epsilon_2 + \epsilon_3$.
		
		On the other hand, the tableau
		\begin{adjustbox}{valign=c}
			\begin{ytableau}
				2 & 3 \\
				3 & \bar{3} \\
				\bar{3} & \bar{2}
			\end{ytableau}
		\end{adjustbox}
		does not belong to $\mathrm{KN}_{\lambda}(6)$, because
		\[
		(q - p)+(s - r)=(1 -1)+(3 -2)=1 \not<1 =j - i =3 -2.
		\]
		
	\end{example}

	\section{Kashiwara-Nakashima Tableaux of $\mathfrak{sp}_4$}
	For the symplectic Lie algebra $\mathfrak{sp}_4$, let $\mathcal{N}=\{1, 2, \overline{2}, \overline{1}\}$ with the total order $1 < 2 < \overline{2} < \overline{1}$. Let $\lambda$ be a partition with $\ell(\lambda) \leq 2$. According to Definition~\ref{KN}, it is straightforward to see that if $T \in \mathrm{KN}_{\lambda}(4)$, then $T$ is a filling of the Young diagram of shape $\lambda$ with entries from $\mathcal{N}$, satisfying the following conditions:
	\begin{enumerate}[label=(\roman*)]
		\item The entries in $T$ are weakly increasing across each row and strictly increasing down each column.
		\item The entries \(1\) and \(\overline{1}\) do not appear in the same column.
		\item If \(T\) has two adjacent columns, then the following two configurations are forbidden:
		$$\left.\begin{array}{lr} 
			& 2 \\
			\\
		\end{array}
		\right|  \begin{array}{r}  
			\cr 2 \cr \overline{2} \cr \cr  
		\end{array}\shiftnoarrow{-0.2}{,} \quad \quad
		\left.\begin{array}{r}
			2  \\
			\overline{2} \\
		\end{array}
		\right|  \begin{array}{r}  
			\cr \overline{2} \cr
		\end{array}\shiftnoarrow{-0.2}{.}$$
	\end{enumerate}
	
	Let \( L(\lambda) \) be the finite dimensional irreducible representation of \( \mathfrak{sp}_4 \) with highest weight \( \lambda \), where
	\[
	\lambda = m_1 \omega_1 + m_2 \omega_2 = (m_1 + m_2)\epsilon_1 + m_2 \epsilon_2, \quad m_1, m_2 \in \mathbb{Z}_{\geq 0}.
	\]
	
	\begin{proposition}\label{prop:dimension}
		$\dim L(\lambda) = |\mathrm{KN}_{\lambda}(4)|.$
	\end{proposition}
	\begin{proof}
		As shown in \cite{kashiwara1994crystal, hong2002introduction}, the KN tableaux of \( \mathfrak{sp}_4 \) were developed to label elements of the crystal basis of $U_q(\mathfrak{sp}_4)$-module $L_q(\lambda)$, which is quantum deformation of $L(\lambda)$. There is a one-to-one correspondence between the elements of the crystal basis and the KN tableaux of \( \mathfrak{sp}_4 \). 
	\end{proof}
	
	Let $\lambda=(m_{1}+m_{2}, m_{2})$ and $T\in \mathrm{KN}_{\lambda}(4)$, where $m_{1}, m_{2}\in \mathbb{Z}_{\geq0}$. Define:
	\begin{enumerate}[label=(\roman*), itemsep=0pt]
		\item $a_1$ as the number of entries in the second row of $T$ that are strictly greater than $2$;
		\item $a_2$ as the sum of the number of entries in the first row strictly greater than $1$ and the number of entries in the second row strictly greater than $\overline{2}$;
		\item $a_3$ as the number of entries in the first row strictly greater than $2$;
		\item $a_4$ as the number of entries in the first row strictly greater than $\overline{2}$.
	\end{enumerate}
	From the definitions of $a_i$ ($i=1,2,3,4$), we can first obtain the following inequalities:
	\[ 0\leq a_{1}\leq m_{2}, \quad 0\leq a_{2}\leq m_{1}+2a_{1}, \quad 0\leq a_{3}\leq a_{2}, \quad 0\leq a_{4}\leq \min\{m_{1}, a_{3}\}. \]
	
	According to the definition of \(\mathrm{KN}_{\lambda}(4)\), the tableau \(T\) falls into one of the following three cases, shown in Figures~\ref{KN1}, \ref{KN2}, and \ref{KN3}.

	\begin{figure}[H]
		\begin{center}
			\begin{tikzpicture}[scale=0.7, every node/.style={font=\large}]
				
				\draw (0,0) rectangle ++(3,1);
				\node at (0.5,0.5) {1};
				\node at (1.5,0.5) {$\cdots$};
				\node at (2.5,0.5) {1};
				
				\draw (3,0) rectangle ++(4,1);
				\node at (3.5,0.5) {1};
				\node at (4.5,0.5) {$\cdots$};
				\node at (5.5,0.5) {1};
				\node at (6.5,0.5) {2};

				\draw (7,0) rectangle ++(8,1);
				\node at (7.5,0.5) {2};
				\node at (8.4,0.5) {$\cdots$};
				\node at (9.3,0.5) {2};
				\node at (10.2,0.5) {$\overline{2}$};
				\node at (11.1,0.5) {$\cdots$};
				\node at (12,0.5) {$\overline{2}$};
				\node at (12.9,0.5) {$\overline{1}$};
				\node at (13.8,0.5) {$\cdots$};
				\node at (14.7,0.5) {$\overline{1}$};

				\draw (0,-1) rectangle ++(3,1);
				\node at (0.5,-0.5) {2};
				\node at (1.5,-0.5) {$\cdots$};
				\node at (2.5,-0.5) {2};
				
				\draw (3,-1) rectangle ++(4,1);
				\node at (3.5,-0.5) {$\overline{2}$};
				\node at (4.5,-0.5) {$\cdots$};
				\node at (5.5,-0.5) {$\overline{2}$};
				\node at (6.5,-0.5) {$\overline{2}$};
				
				\draw (7,-1) rectangle ++(5,1);
				\node at (8,-0.5) {$\overline{1}$};
				\node at (9.5,-0.5) {$\cdots$};
				\node at (11,-0.5) {$\overline{1}$};

			\end{tikzpicture}
			\caption{} 
			\label{KN1}
		\end{center}
	\end{figure}
	
	\begin{figure}[H]
		\begin{center}
			\begin{tikzpicture}[scale=0.7, every node/.style={font=\large}]
				
				\draw (0,0) rectangle ++(3,1);
				\node at (0.5,0.5) {1};
				\node at (1.5,0.5) {$\cdots$};
				\node at (2.5,0.5) {1};
				
				\draw (3,0) rectangle ++(4,1);
				\node at (3.5,0.5) {1};
				\node at (4.5,0.5) {$\cdots$};
				\node at (5.5,0.5) {1};
				\node at (6.5,0.5) {1};
				
				\draw (7,0) rectangle ++(8,1);
				\node at (7.5,0.5) {2};
				\node at (8.4,0.5) {$\cdots$};
				\node at (9.3,0.5) {2};
				\node at (10.2,0.5) {$\overline{2}$};
				\node at (11.1,0.5) {$\cdots$};
				\node at (12,0.5) {$\overline{2}$};
				\node at (12.9,0.5) {$\overline{1}$};
				\node at (13.8,0.5) {$\cdots$};
				\node at (14.7,0.5) {$\overline{1}$};
				
				\draw (0,-1) rectangle ++(3,1);
				\node at (0.5,-0.5) {2};
				\node at (1.5,-0.5) {$\cdots$};
				\node at (2.5,-0.5) {2};
				
				\draw (3,-1) rectangle ++(4,1);
				\node at (3.5,-0.5) {$\overline{2}$};
				\node at (4.5,-0.5) {$\cdots$};
				\node at (5.5,-0.5) {$\overline{2}$};
				\node at (6.5,-0.5) {$\overline{2}$};
				
				\draw (7,-1) rectangle ++(5,1);
				\node at (8,-0.5) {$\overline{1}$};
				\node at (9.5,-0.5) {$\cdots$};
				\node at (11,-0.5) {$\overline{1}$};

			\end{tikzpicture}
			\caption{} 
			\label{KN2}
		\end{center}
	\end{figure}
	
	\begin{figure}[H]
		\begin{center}
			\begin{tikzpicture}[scale=0.7, every node/.style={font=\large}]
				
				\draw (0,0) rectangle ++(3,1);
				\node at (0.5,0.5) {1};
				\node at (1.5,0.5) {$\cdots$};
				\node at (2.5,0.5) {1};
				
				\draw (3,0) rectangle ++(3,1);
				\node at (3.5,0.5) {1};
				\node at (4.5,0.5) {$\cdots$};
				\node at (5.5,0.5) {1};
				
				\draw (6,0) rectangle ++(9,1);
				\node at (6.5,0.5) {1};
				\node at (7.2,0.5) {$\cdots$};
				\node at (7.9,0.5) {1};
				\node at (8.6,0.5) {2};
				\node at (9.3,0.5) {$\cdots$};
				\node at (10,0.5) {2};
				\node at (10.8,0.5) {$\overline{2}$};
				\node at (11.5,0.5) {$\cdots$};
				\node at (12.2,0.5) {$\overline{2}$};
				\node at (13,0.5) {$\overline{1}$};
				\node at (13.7,0.5) {$\cdots$};
				\node at (14.4,0.5) {$\overline{1}$};
				
				\draw (0,-1) rectangle ++(3,1);
				\node at (0.5,-0.5) {2};
				\node at (1.5,-0.5) {$\cdots$};
				\node at (2.5,-0.5) {2};
				
				\draw (3,-1) rectangle ++(3,1);
				\node at (3.5,-0.5) {$\overline{2}$};
				\node at (4.5,-0.5) {$\cdots$};
				\node at (5.5,-0.5) {$\overline{2}$};
				
			\end{tikzpicture}
			\caption{} 
			\label{KN3}
		\end{center}
	\end{figure}
	
	If \(T\) corresponds to Figure~\ref{KN1}, then 
	\[
	0 \leq a_{3} \leq \frac{1}{2}(a_{2} + m_{1} - 1) = \lfloor \tfrac{1}{2}(a_{2} + m_{1}) \rfloor \leq a_{2},
	\]
	where \(\lfloor \cdot \rfloor\) denotes the floor function.  
	
	If \(T\) corresponds to Figure~\ref{KN2}, then 
	\[
	0 \leq a_{3} \leq \frac{1}{2}(a_{2} + m_{1}) \leq a_{2}.
	\]
	
	If \(T\) corresponds to Figure~\ref{KN3}, then 
	\[
	0 \leq a_{3} \leq a_{2} \leq \frac{1}{2}(a_{2} + m_{1}).
	\]
	
	Since \(a_{3} \in \mathbb{Z}_{\geq 0}\), we have the combined condition:
	\[
	0 \leq a_{3} \leq \min\left\{\frac{1}{2}(a_{2} + m_{1}), \, a_{2}\right\}.
	\]

	In summary, we have the following system of inequalities:
	\begin{equation} \label{ineq}
		\begin{aligned}
			0 &\leq a_{1} \leq m_{2}, \\
			0 &\leq a_{2} \leq m_{1} + 2a_{1}, \\
			0 &\leq a_{3} \leq \min\left\{\tfrac{1}{2}(a_{2} + m_{1}),\, a_{2}\right\}, \\
			0 &\leq a_{4} \leq \min\{m_{1},\, a_{3}\}.
		\end{aligned}
	\end{equation}
	
	\section{Verma Vectors of Irreducible Representations of $\mathfrak{sp}_4$}
	The system of inequalities obtained in the previous section will be directly used in this section to construct the Verma vectors of the finite dimensional irreducible representation of $\mathfrak{sp}_4$, and to prove that there exists a weight-preserving one-to-one correspondence between the Verma vectors and the KN tableaux of $\mathfrak{sp}_4$.
	
	Let $L(\lambda)$ be the finite dimensional irreducible representation of $\mathfrak{sp}_4$ with highest weight $\lambda$, where $\lambda=m_{1}\omega_{1}+m_{2}\omega_{2}=(m_{1}+m_{2})\epsilon_{1}+m_{2}\epsilon_{2}$. Let the non-zero vector $v_{\lambda}$ be the highest weight vector of $L(\lambda)$.
	
	\begin{definition}
		The weight vector
		\begin{equation}\label{Verma vector}
			f_1^{a_4} f_2^{a_3} f_1^{a_2} f_2^{a_1}  v_{\lambda}
		\end{equation}
		is called a Verma vector of the irreducible representation $L(\lambda)$ of $\mathfrak{sp}_4$, where the parameters satisfy
		\begin{equation} \label{verma inequalities}
			\begin{aligned}
				a_i &\in \mathbb Z_{\ge0}, \;i=1,2,3,4,\\
				0 &\leq a_1 \leq m_2, \\
				0 &\leq a_2 \leq m_1 + 2a_1, \\
				0 &\leq a_3 \leq \min\left\{ \tfrac{1}{2}(a_2 + m_1),\ a_2 \right\}, \\
				0 &\leq a_4 \leq \min\{ m_1,\ a_3 \}.
			\end{aligned}
		\end{equation}
		Here $f_{1}=E_{21}-E_{34}$ is the root vector corresponding to the negative simple root $\epsilon_2-\epsilon_1$ of $\mathfrak{sp}_4$, and $f_{2}=E_{42}$ is the root vector corresponding to the negative simple root $-2\epsilon_2$.
	\end{definition}

	We denote the Verma vector $ f_1^{a_4} f_2^{a_3} f_1^{a_2} f_2^{a_1} v_\lambda $ by $ \bm{f^{a}}v_\lambda $, where $ \bm{a} = (a_1, a_2, a_3, a_4) $. Let $ F $ denote the set of all Verma vectors of $L(\lambda)$. Then we have the following theorem.
	
	\begin{theorem}\label{KN-correspondence}
		There exists a one-to-one correspondence between the set $F$ of Verma vectors of the irreducible representation $L(\lambda)$ of $\mathfrak{sp}_{4}$ and the set $\mathrm{KN}_{\lambda}(4)$ of KN tableaux of $\mathfrak{sp}_4$.
	\end{theorem}
	\begin{proof}
		Construct a mapping
		\begin{equation*}
			\begin{aligned}
				\varphi \colon \mathrm{KN}_{\lambda}(4) &\longrightarrow F \\
				T &\longmapsto \bm{f^{a}}v_{\lambda}, \quad \forall T \in \mathrm{KN}_{\lambda}(4),
			\end{aligned}
		\end{equation*}
		where $\bm{a}=(a_{1}, a_{2}, a_{3}, a_{4})$, $a_{1}$ is the number of entries in the second row of $T$ strictly greater than $2$, $a_{2}$ is the sum of the number of entries in the first row strictly greater than $1$ and the number of entries in the second row strictly greater than $\overline{2}$, $a_{3}$ is the number of entries in the first row strictly greater than $2$, and $a_{4}$ is the number of entries in the first row strictly greater than $\overline{2}$. Thus, $\bm{a}=(a_{1}, a_{2}, a_{3}, a_{4})$ satisfies the system of inequalities \eqref{verma inequalities}, and hence $\bm{f^{a}}v_{\lambda} \in F$. Therefore, $\varphi$ is well-defined.
		
		For any $\bm{f^{a}}v_{\lambda} \in F$ with $\bm{a}=(a_{1}, a_{2}, a_{3}, a_{4})$, we need to prove that there exists a unique $T(\bm{a})\in \mathrm{KN}_{\lambda}(4)$ such that $\varphi (T(\bm{a}))=\bm{f^{a}}v_{\lambda}$. Specifically, we need to discuss the cases for $\bm{a}=(a_{1}, a_{2}, a_{3}, a_{4})$ as follows:
		
		If $a_{2}\geq m_{1}$ and $\frac{1}{2}(a_{2}-m_{1}-1)\in \mathbb{Z}_{\geq0}$, then the unique $T(\bm{a})$ such that $\varphi(T(\bm{a}))=\bm{f^{a}}v_{\lambda}$ is shown in Figure~\ref{KN 1}. We only use this case as an example to illustrate the uniqueness; other cases can be shown similarly. Since the number of entries strictly greater than $2$ in the second row of $T(\bm{a})$ must be $a_{1}$, the number of $2$'s in the second row of $T(\bm{a})$ can only be $m_2-a_1$; since the sum of the number of entries in the first row strictly greater than $1$ and the number of entries in the second row strictly greater than $\overline{2}$ must be $a_{2}$, combined with the conditions of this case, we can obtain that the number of $\overline{1}$'s in the second row of $T(\bm{a})$ can only be $\frac{1}{2}(a_{2}-m_{1}-1)$, and thus the number of $\overline{2}$'s in the second row of $T(\bm{a})$ is $a_{1}-\frac{1}{2}(a_{2}-m_{1}-1)$; since the number of entries strictly greater than $\overline{2}$ in the first row of $T(\bm{a})$ must be $a_{4}$, the number of $\overline{1}$'s in the first row of $T(\bm{a})$ can only be $a_{4}$; moreover, since the number of entries strictly greater than $2$ in the first row of $T(\bm{a})$ must be $a_{3}$, the number of $\overline{2}$'s in the first row of $T(\bm{a})$ can only be $a_3-a_{4}$, from which we can deduce that the number of $2$'s in the first row of $T(\bm{a})$ is $m_1-a_3+\frac{1}{2}(a_{2}-m_{1}-1)+1$, and the number of $1$'s in the first row of $T(\bm{a})$ is $m_{2}-\frac{1}{2}(a_{2}-m_{1}-1)-1$. Therefore, there is one and only one $T(\bm{a})$ satisfying $\varphi(T(\bm{a}))=\bm{f^{a}}v_{\lambda}$, as shown in Figure~\ref{KN 1}.
		\begin{figure}[H]
			\begin{center}
				\begin{tikzpicture}[scale=0.8,
					every node/.style={font=\large},
					brace/.style={decorate,decoration={brace,  amplitude=5pt,raise=2pt}},
					brace mirror/.style={  decorate,decoration={brace,  mirror, amplitude=5pt,raise=2pt}},
					label/.style={midway,font=\small,outer sep=6pt} ]
					
					\draw (0,0) rectangle (3,1);
					\foreach \x in {0.5,2.5} \node at (\x,0.5) {1};
					\node at (1.5,0.5) {$\cdots$};
					
					\draw (3,0) rectangle (7,1);
					\foreach \x in {3.5,5.5} \node at (\x,0.5) {1};
					\node at (4.5,0.5) {$\cdots$};
					\node at (6.5,0.5) {2};
					
					\draw[brace] (0,1.1) -- node[label,above,xshift=-9mm] {$m_{2}-\frac{1}{2}(a_{2}-m_{1}-1)-1$} (5.5,1.1);
					
					\draw (7,0) rectangle (16,1);
					\foreach \x in {7.5,9.5} \node at (\x,0.5) {2};
					\node at (8.5,0.5) {$\cdots$};
					\draw[brace] (6.4,1.1) -- node[label,above,xshift=-3mm] {$m_1-a_3+\frac{1}{2}(a_{2}-m_{1}-1)+1$} (9.6,1.1);
					
					\foreach \x in {10.5,12.5} \node at (\x,0.5) {$\overline{2}$};
					\node at (11.5,0.5) {$\cdots$};
					\draw[brace] (10.3,1.1) -- node[label,above,xshift=3mm] {$a_{3}-a_{4}$} (12.6,1.1);
					
					\foreach \x in {13.5,15.5} \node at (\x,0.5) {$\overline{1}$};
					\node at (14.5,0.5) {$\cdots$};
					\draw[brace] (13.3,1.1) -- node[label,above] {$a_{4}$} (16,1.1);

					\draw (0,-1) rectangle (3,0);
					\foreach \x in {0.5,2.5} \node at (\x,-0.5) {2};
					\node at (1.5,-0.5) {$\cdots$};
					\draw[brace mirror] (0,-1.1) -- node[label,below] {$m_{2}-a_1$} (3,-1.1);
					
					\draw (3,-1) rectangle (7,0);
					
					\foreach \x in {3.5,5.5,6.5} \node at (\x,-0.5) {$\overline{2}$};
					\node at (4.5,-0.5) {$\cdots$};
					\draw[brace mirror] (3,-1.1) -- node[label,below] {$a_{1}-\frac{1}{2}(a_{2}-m_{1}-1)$} (7,-1.1);
					
					\draw (7,-1) rectangle (12,0);
					\foreach \x in {8,11} \node at (\x,-0.5) {$\overline{1}$};
					\node at (9.5,-0.5) {$\cdots$};
					\draw[brace mirror] (7,-1.1) -- node[label,below] {$\frac{1}{2}(a_{2}-m_{1}-1)$} (12,-1.1);
					
				\end{tikzpicture}
				\caption{}
				\label{KN 1}
			\end{center}
		\end{figure}
		
		If $a_{2}\geq m_{1}$ and $\frac{1}{2}(a_{2}-m_{1})\in \mathbb{Z}_{\geq0}$, then the unique $T(\bm{a})$ such that $\varphi(T(\bm{a}))=\bm{f^{a}}v_{\lambda}$ is shown in Figure~\ref{KN 2}. 
		\begin{figure}[H]
			\begin{center}
				\begin{tikzpicture}[scale=0.8,
					every node/.style={font=\large},
					brace/.style={decorate,decoration={brace,  amplitude=5pt,raise=2pt}},
					brace mirror/.style={  decorate,decoration={brace,  mirror, amplitude=5pt,raise=2pt}},
					label/.style={midway,font=\small,outer sep=6pt} ]
					
					\draw (0,0) rectangle (3,1);
					\foreach \x in {0.5,2.5} \node at (\x,0.5) {1};
					\node at (1.5,0.5) {$\cdots$};
					
					\draw (3,0) rectangle (7,1);
					\foreach \x in {3.5,5.5,6.5} \node at (\x,0.5) {1};
					\node at (4.5,0.5) {$\cdots$};
					
					\draw[brace] (0,1.1) -- node[label,above,xshift=-3mm] {$m_{2}-\frac{1}{2}(a_{2}-m_{1})$} (7,1.1);
					
					\draw (7,0) rectangle (16,1);
					\foreach \x in {7.5,9.5} \node at (\x,0.5) {2};
					\node at (8.5,0.5) {$\cdots$};
					\draw[brace] (7,1.1) -- node[label,above,xshift=-3mm] {$m_1-a_3+\frac{1}{2}(a_{2}-m_{1})$} (9.6,1.1);
					
					\foreach \x in {10.5,12.5} \node at (\x,0.5) {$\overline{2}$};
					\node at (11.5,0.5) {$\cdots$};
					\draw[brace] (10.3,1.1) -- node[label,above,xshift=2mm] {$a_{3}-a_{4}$} (12.6,1.1);
					
					\foreach \x in {13.5,15.5} \node at (\x,0.5) {$\overline{1}$};
					\node at (14.5,0.5) {$\cdots$};
					\draw[brace] (13.3,1.1) -- node[label,above] {$a_{4}$} (16,1.1);

					\draw (0,-1) rectangle (3,0);
					\foreach \x in {0.5,2.5} \node at (\x,-0.5) {2};
					\node at (1.5,-0.5) {$\cdots$};
					\draw[brace mirror] (0,-1.1) -- node[label,below] {$m_{2}-a_{1}$} (3,-1.1);
					
					\draw (3,-1) rectangle (7,0);
					
					\foreach \x in {3.5,5.5,6.5} \node at (\x,-0.5) {$\overline{2}$};
					\node at (4.5,-0.5) {$\cdots$};
					\draw[brace mirror] (3,-1.1) -- node[label,below] {$a_{1}-\frac{1}{2}(a_{2}-m_{1})$} (7,-1.1);
					
					\draw (7,-1) rectangle (12,0);
					\foreach \x in {8,11} \node at (\x,-0.5) {$\overline{1}$};
					\node at (9.5,-0.5) {$\cdots$};
					\draw[brace mirror] (7,-1.1) -- node[label,below] {$\frac{1}{2}(a_{2}-m_{1})$} (12,-1.1);
					
				\end{tikzpicture}
				\caption{}
				\label{KN 2}
			\end{center}
		\end{figure}
		If $a_{2}< m_{1}$, then the unique $T(\bm{a})$ such that $\varphi(T(\bm{a}))=\bm{f^{a}}v_{\lambda}$ is shown in Figure~\ref{KN 3}. 
		\begin{figure}[H]
			\begin{center}
				\begin{tikzpicture}[scale=0.8,
					every node/.style={font=\large},
					brace/.style={decorate,decoration={brace,  amplitude=5pt,raise=2pt}},
					brace mirror/.style={  decorate,decoration={brace,  mirror, amplitude=5pt,raise=2pt}},
					label/.style={midway,font=\small,outer sep=6pt} ]
					
					\draw (0,0) rectangle (3,1);
					\foreach \x in {0.5,2.5} \node at (\x,0.5) {1};
					\node at (1.5,0.5) {$\cdots$};
					
					\draw (3,0) rectangle (6,1);
					\foreach \x in {3.5,5.5,6.5,8.5} \node at (\x,0.5) {1};
					\node at (4.5,0.5) {$\cdots$};
					\node at (7.5,0.5) {$\cdots$};
					
					\draw[brace] (0,1.1) -- node[label,above] {$m_{1}+m_{2}-a_2$} (8.5,1.1);
					
					\draw (6,0) rectangle (18,1);
					\foreach \x in {9.5,11.5} \node at (\x,0.5) {2};
					\node at (10.5,0.5) {$\cdots$};
					\draw[brace] (9.4,1.1) -- node[label,above] {$a_{2}-a_3$} (11.6,1.1);
					
					\foreach \x in {12.5,14.5} \node at (\x,0.5) {$\overline{2}$};
					\node at (13.5,0.5) {$\cdots$};
					\draw[brace] (12.3,1.1) -- node[label,above] {$a_{3}-a_{4}$} (14.6,1.1);
					
					\foreach \x in {15.5,17.5} \node at (\x,0.5) {$\overline{1}$};
					\node at (16.5,0.5) {$\cdots$};
					\draw[brace] (15.3,1.1) -- node[label,above] {$a_{4}$} (18,1.1);

					\draw (0,-1) rectangle (3,0);
					\foreach \x in {0.5,2.5} \node at (\x,-0.5) {2};
					\node at (1.5,-0.5) {$\cdots$};
					\draw[brace mirror] (0,-1.1) -- node[label,below] {$m_{2}-a_{1}$} (3,-1.1);
					
					\draw (3,-1) rectangle (6,0);
					
					\foreach \x in {3.5,5.5} \node at (\x,-0.5) {$\overline{2}$};
					\node at (4.5,-0.5) {$\cdots$};
					\draw[brace mirror] (3,-1.1) -- node[label,below] {$a_{1}$} (6,-1.1);
					
				\end{tikzpicture}
				\caption{}
				\label{KN 3}
			\end{center}
		\end{figure}
		Therefore, $\varphi$ is a bijection, and the conclusion holds.
	\end{proof}

	Let \( T(\bm{a}) \) denote the KN tableau of $\mathfrak{sp}_4$ corresponding to the Verma vector \( \bm{f^{a}}v_{\lambda} \). Then we have the following theorem.
	
	\begin{theorem}
		The weight of the Verma vector \( \bm{f^{a}}v_{\lambda} \) is equal to the weight of its corresponding KN tableau \( T(\bm{a}) \) of $\mathfrak{sp}_4$.
	\end{theorem}
	
	\begin{proof}
		Let \( \bm{a} = (a_{1}, a_{2}, a_{3}, a_{4}) \), and denote the weight of \( \bm{f^{a}}v_{\lambda} \) by \( \mu \). Then
		\begin{align*}
			\mu &=\lambda-a_{1}(2\epsilon_{2})-a_{2}(\epsilon_{1}-\epsilon_{2})-a_{3}(2\epsilon_{2})-a_{4}(\epsilon_{1}-\epsilon_{2}) \\&=(m_{1}+m_{2})\epsilon_{1}+m_{2}\epsilon_{2}-a_{1}(2\epsilon_{2})-a_{2}(\epsilon_{1}-\epsilon_{2})-a_{3}(2\epsilon_{2})-a_{4}(\epsilon_{1}-\epsilon_{2})\\&=(m_{1}+m_{2}-a_{2}-a_{4})\epsilon_{1}+(m_{2}-2a_{1}+a_{2}-2a_{3}+a_{4})\epsilon_{2}.
		\end{align*}
		Next, we calculate the weight of \( T(\bm{a}) \) according to the cases for \( \bm{a} = (a_{1}, a_{2}, a_{3}, a_{4}) \).
		
		If \( a_{2} \geq m_{1} \) and \( \frac{1}{2}(a_{2} - m_{1} - 1) \in \mathbb{Z}_{\geq 0} \), then \( T(\bm{a}) \) is as shown in Figure~\ref{KN 1}. It follows that
		\begin{align*}
			\mathrm{wt}(T(\bm{a}))&=\Big(m_{2}-\tfrac{1}{2}(a_{2}-m_{1}-1)-1-a_{4}-\tfrac{1}{2}(a_{2}-m_{1}-1)\Big)\epsilon_{1}\\&\quad+\Big(m_{1}-a_{3}+\tfrac{1}{2}(a_{2}-m_{1}-1)+1+m_{2}-a_{1}\\&\quad-(a_{3}-a_{4})-(a_{1}-\tfrac{1}{2}(a_{2}-m_{1}-1))\Big)\epsilon_{2}\\&=(m_{1}+m_{2}-a_{2}-a_{4})\epsilon_{1}+(m_{2}-2a_{1}+a_{2}-2a_{3}+a_{4})\epsilon_{2}.
		\end{align*}
		
		If \( a_{2} \geq m_{1} \) and \( \frac{1}{2}(a_{2} - m_{1}) \in \mathbb{Z}_{\geq 0} \), then \( T(\bm{a}) \) is as shown in Figure~\ref{KN 2}. It follows that
		\begin{align*}
			\mathrm{wt}(T(\bm{a}))&=\Big(m_{2}-\tfrac{1}{2}(a_{2}-m_{1})-a_{4}-\tfrac{1}{2}(a_{2}-m_{1})\Big)\epsilon_{1}\\&\quad+\Big(m_{1}-a_{3}+\tfrac{1}{2}(a_{2}-m_{1})+m_{2}-a_{1}\\&\quad-(a_{3}-a_{4})-(a_{1}-\tfrac{1}{2}(a_{2}-m_{1}))\Big)\epsilon_{2}\\&=(m_{1}+m_{2}-a_{2}-a_{4})\epsilon_{1}+(m_{2}-2a_{1}+a_{2}-2a_{3}+a_{4})\epsilon_{2}.
		\end{align*}
		
		If \( a_{2} < m_{1} \), then \( T(\bm{a}) \) is as shown in Figure~\ref{KN 3}. It follows that
		\begin{align*}
			\mathrm{wt}(T(\bm{a}))&=\Big(m_{1}+m_{2}-a_{2}-a_{4}\Big)\epsilon_{1}+\Big(m_{2}-a_{1}+a_{2}-a_{3}-a_{1}-(a_{3}-a_{4})\Big)\epsilon_{2}\\&=(m_{1}+m_{2}-a_{2}-a_{4})\epsilon_{1}+(m_{2}-2a_{1}+a_{2}-2a_{3}+a_{4})\epsilon_{2}.
		\end{align*}
		
		Therefore, \( \mathrm{wt}(T(\bm{a})) = \mu \).
	\end{proof}
	
	\begin{example}\label{example}
		Let the highest weight of the irreducible representation $L(\lambda)$ of $\mathfrak{sp}_4$ be $\lambda = \omega_{1} + 2\omega_{2} = 3\epsilon_{1} + 2\epsilon_{2}$. The Verma vectors of $L(\lambda)$ have the form
		$
		f_1^{a_4} f_2^{a_3} f_1^{a_2} f_2^{a_1} v_{\lambda},
		$
		where the parameters satisfy
		$$
		0 \leq a_{1} \leq 2,\quad 
		0 \leq a_{2} \leq 1 + 2a_{1},\quad 
		0 \leq a_{3} \leq \min\left\{\tfrac{1}{2}(a_{2} + 1),\ a_{2}\right\},\quad 
		0 \leq a_{4} \leq \min\{1,\ a_{3}\}.
		$$
		
		These Verma vectors can be represented by KN tableaux of $\mathfrak{sp}_4$ as follows:
		\begin{center}
			\renewcommand{\arraystretch}{2.5}
			\setlength{\tabcolsep}{12pt}
			\begin{longtable}{ccccc}
				\YoungWithData{
					\begin{ytableau}
						1 & 1 & 1 \\
						2 & 2
					\end{ytableau}
				}{v_{\lambda}}
				&
				\YoungWithData{
					\begin{ytableau}
						1 & 1 & 2 \\
						2 & 2
					\end{ytableau}
				}{f_{1}v_{\lambda}}
				&
				\YoungWithData{
					\begin{ytableau}
						1 & 1 & \overline{2} \\
						2 & 2
					\end{ytableau}
				}{f_{2}f_{1}v_{\lambda}}
				&
				\YoungWithData{
					\begin{ytableau}
						1 & 1 & \overline{1} \\
						2 & 2
					\end{ytableau}
				}{f_{1}f_{2}f_{1}v_{\lambda}}
				&
				\YoungWithData{
					\begin{ytableau}
						1 & 1 & 1 \\
						2 & \overline{2}
					\end{ytableau}
				}{f_{2}v_{\lambda}} \\
				
				\YoungWithData{
					\begin{ytableau}
						1 & 1 & 2 \\
						2 & \overline{2}
					\end{ytableau}
				}{f_{1}f_{2}v_{\lambda}}
				&
				\YoungWithData{
					\begin{ytableau}
						1 & 1 & \overline{2} \\
						2 & \overline{2}
					\end{ytableau}
				}{f_{2}f_{1}f_{2}v_{\lambda}}
				&
				\YoungWithData{
					\begin{ytableau}
						1 & 1 & \overline{1} \\
						2 & \overline{2}
					\end{ytableau}
				}{f_{1}f_{2}f_{1}f_{2}v_{\lambda}}
				&
				\YoungWithData{
					\begin{ytableau}
						1 & 2 & 2 \\
						2 & \overline{2}
					\end{ytableau}
				}{f_{1}^{2}f_{2}v_{\lambda}}
				&
				\YoungWithData{
					\begin{ytableau}
						1 & 2 & \overline{2} \\
						2 & \overline{2}
					\end{ytableau}
				}{f_{2}f_{1}^{2}f_{2}v_{\lambda}} \\
				
				\YoungWithData{
					\begin{ytableau}
						1 & 2 & \overline{1} \\
						2 & \overline{2}
					\end{ytableau}
				}{f_{1}f_{2}f_{1}^{2}f_{2}v_{\lambda}}
				&
				\YoungWithData{
					\begin{ytableau}
						1 & 2 & 2 \\
						2 & \overline{1}
					\end{ytableau}
				}{f_{1}^{3}f_{2}v_{\lambda}}
				&
				\YoungWithData{
					\begin{ytableau}
						1 & 2 & \overline{2} \\
						2 & \overline{1}
					\end{ytableau}
				}{f_{2}f_{1}^{3}f_{2}v_{\lambda}}
				&
				\YoungWithData{
					\begin{ytableau}
						1 & 2 & \overline{1} \\
						2 & \overline{1}
					\end{ytableau}
				}{f_{1}f_{2}f_{1}^{3}f_{2}v_{\lambda}}
				&
				\YoungWithData{
					\begin{ytableau}
						1 & \overline{2} & \overline{2} \\
						2 & \overline{1}
					\end{ytableau}
				}{f_{2}^{2}f_{1}^{3}f_{2}v_{\lambda}} \\
				
				\YoungWithData{
					\begin{ytableau}
						1 & \overline{2} & \overline{1} \\
						2 & \overline{1}
					\end{ytableau}
				}{f_{1}f_{2}^{2}f_{1}^{3}f_{2}v_{\lambda}}
				&
				\YoungWithData{
					\begin{ytableau}
						1 & 1 & 1 \\
						\overline{2} & \overline{2}
					\end{ytableau}
				}{f_{2}^{2}v_{\lambda}}
				&
				\YoungWithData{
					\begin{ytableau}
						1 & 1 & 2 \\
						\overline{2} & \overline{2}
					\end{ytableau}
				}{f_{1}f_{2}^{2}v_{\lambda}}
				&
				\YoungWithData{
					\begin{ytableau}
						1 & 1 & \overline{2} \\
						\overline{2} & \overline{2}
					\end{ytableau}
				}{f_{2}f_{1}f_{2}^{2}v_{\lambda}}
				&
				\YoungWithData{
					\begin{ytableau}
						1 & 1 & \overline{1} \\
						\overline{2} & \overline{2}
					\end{ytableau}
				}{f_{1}f_{2}f_{1}f_{2}^{2}v_{\lambda}} \\
				
				\YoungWithData{
					\begin{ytableau}
						1 & 2 & 2 \\
						\overline{2} & \overline{2}
					\end{ytableau}
				}{f_{1}^{2}f_{2}^{2}v_{\lambda}}
				&
				\YoungWithData{
					\begin{ytableau}
						1 & 2 & \overline{2} \\
						\overline{2} & \overline{2}
					\end{ytableau}
				}{f_{2}f_{1}^{2}f_{2}^{2}v_{\lambda}}
				&
				\YoungWithData{
					\begin{ytableau}
						1 & 2 & \overline{1} \\
						\overline{2} & \overline{2}
					\end{ytableau}
				}{f_{1}f_{2}f_{1}^{2}f_{2}^{2}v_{\lambda}}
				&
				\YoungWithData{
					\begin{ytableau}
						1 & 2 & 2 \\
						\overline{2} & \overline{1}
					\end{ytableau}
				}{f_{1}^{3}f_{2}^{2}v_{\lambda}}
				&
				\YoungWithData{
					\begin{ytableau}
						1 & 2 & \overline{2} \\
						\overline{2} & \overline{1}
					\end{ytableau}
				}{f_{2}f_{1}^{3}f_{2}^{2}v_{\lambda}} \\
				
				\YoungWithData{
					\begin{ytableau}
						1 & 2 & \overline{1} \\
						\overline{2} & \overline{1}
					\end{ytableau}
				}{f_{1}f_{2}f_{1}^{3}f_{2}^{2}v_{\lambda}}
				& 
				\YoungWithData{
					\begin{ytableau}
						1 & \overline{2} & \overline{2} \\
						\overline{2} & \overline{1}
					\end{ytableau}
				}{f_{2}^{2}f_{1}^{3}f_{2}^{2}v_{\lambda}}
				& 
				\YoungWithData{
					\begin{ytableau}
						1 & \overline{2} & \overline{1} \\
						\overline{2} & \overline{1}
					\end{ytableau}
				}{f_{1}f_{2}^{2}f_{1}^{3}f_{2}^{2}v_{\lambda}}
				&
				\YoungWithData{
					\begin{ytableau}
						2 & 2 & 2 \\
						\overline{2} & \overline{1}
					\end{ytableau}
				}{f_{1}^{4}f_{2}^{2}v_{\lambda}}
				&
				\YoungWithData{
					\begin{ytableau}
						2 & 2 & \overline{2} \\
						\overline{2} & \overline{1}
					\end{ytableau}
				}{f_{2}f_{1}^{4}f_{2}^{2}v_{\lambda}} \\
				
				\YoungWithData{
					\begin{ytableau}
						2 & 2 & \overline{1} \\
						\overline{2} & \overline{1}
					\end{ytableau}
				}{f_{1}f_{2}f_{1}^{4}f_{2}^{2}v_{\lambda}}
				&
				\YoungWithData{
					\begin{ytableau}
						2 & \overline{2} & \overline{2} \\
						\overline{2} & \overline{1}
					\end{ytableau}
				}{f_{2}^{2}f_{1}^{4}f_{2}^{2}v_{\lambda}}
				&
				\YoungWithData{
					\begin{ytableau}
						2 & \overline{2} & \overline{1} \\
						\overline{2} & \overline{1}
					\end{ytableau}
				}{f_{1}f_{2}^{2}f_{1}^{4}f_{2}^{2}v_{\lambda}}
				&
				\YoungWithData{
					\begin{ytableau}
						2 & 2 & 2 \\
						\overline{1} & \overline{1}
					\end{ytableau}
				}{f_{1}^{5}f_{2}^{2}v_{\lambda}}
				&
				\YoungWithData{
					\begin{ytableau}
						2 & 2 & \overline{2} \\
						\overline{1} & \overline{1}
					\end{ytableau}
				}{f_{2}f_{1}^{5}f_{2}^{2}v_{\lambda}} \\
				
				\YoungWithData{
					\begin{ytableau}
						2 & 2 & \overline{1} \\
						\overline{1} & \overline{1}
					\end{ytableau}
				}{f_{1}f_{2}f_{1}^{5}f_{2}^{2}v_{\lambda}}
				&
				\YoungWithData{
					\begin{ytableau}
						2 & \overline{2} & \overline{2} \\
						\overline{1} & \overline{1}
					\end{ytableau}
				}{f_{2}^{2}f_{1}^{5}f_{2}^{2}v_{\lambda}}
				&
				\YoungWithData{
					\begin{ytableau}
						2 & \overline{2} & \overline{1} \\
						\overline{1} & \overline{1}
					\end{ytableau}
				}{f_{1}f_{2}^{2}f_{1}^{5}f_{2}^{2}v_{\lambda}}
				&
				\YoungWithData{
					\begin{ytableau}
						\overline{2} & \overline{2} & \overline{2} \\
						\overline{1} & \overline{1}
					\end{ytableau}
				}{f_{2}^{3}f_{1}^{5}f_{2}^{2}v_{\lambda}}
				&
				\YoungWithData{
					\begin{ytableau}
						\overline{2} & \overline{2} & \overline{1} \\
						\overline{1} & \overline{1}
					\end{ytableau}
				}{f_{1}f_{2}^{3}f_{1}^{5}f_{2}^{2}v_{\lambda}} \\
			\end{longtable}
		\end{center}
	\end{example}
	
	\section{Linear Independence of Verma Vectors of the Representation $L(\lambda)$}
	In this section, we show that the Verma vectors of the finite dimensional irreducible representation $L(\lambda)$ of $\mathfrak{sp}_{4}$ are linearly independent.
	Thus $F$ forms a basis of $L(\lambda)$. For the proof of linear independence, we refer to the method in \cite{Raghavan1999}.
		By the way,
	the Verma basis depends on the expression of $w_0$
	in terms of the simple reflections with shortest length specifically.
	
	We first consider the natural representation $V=\mathbb{C}^{4}$ of $\mathfrak{sp}_{4}$. Let $\{\varepsilon_1, \varepsilon_2, \varepsilon_3, \varepsilon_4\}$ be the standard basis of $V$:
	\[
	\varepsilon_1 = \begin{pmatrix}
		1 \\
		0 \\
		0 \\
		0 
	\end{pmatrix}, \quad
	\varepsilon_2 = \begin{pmatrix}
		0 \\
		1 \\
		0 \\
		0 
	\end{pmatrix}, \quad
	\varepsilon_3 = \begin{pmatrix}
		0 \\
		0 \\
		1 \\
		0 
	\end{pmatrix}, \quad
	\varepsilon_4 = \begin{pmatrix}
		0 \\
		0 \\
		0 \\
		1 
	\end{pmatrix}.
	\]
	
	The highest weight of $V$ is $\epsilon_1$, with highest weight vector $\varepsilon_1$. The Verma vectors are:
	\[
	\varepsilon_1 = \begin{pmatrix}
		1 \\
		0 \\
		0 \\
		0
	\end{pmatrix}, \quad
	f_1\varepsilon_1 = \begin{pmatrix}
		0 \\
		1 \\
		0 \\
		0
	\end{pmatrix}, \quad
	f_2f_1\varepsilon_1 = \begin{pmatrix}
		0 \\
		0 \\
		0 \\
		1
	\end{pmatrix}, \quad
	f_1f_2f_1\varepsilon_1 = \begin{pmatrix}
		0 \\
		0 \\
		-1 \\
		0 
	\end{pmatrix},
	\]
	where $f_{1} = E_{21} - E_{34}$ and $f_{2} = E_{42}$. Thus, the Verma vectors in the natural representation $V$ are linearly independent. 
	Setting $\varepsilon_{\overline{2}} \coloneqq \varepsilon_4$ and $\varepsilon_{\overline{1}} \coloneqq -\varepsilon_3$, the set of Verma vectors of $V$ is $\left\lbrace \varepsilon_1, \varepsilon_2, \varepsilon_{\overline{2}}, \varepsilon_{\overline{1}} \right\rbrace$. Then we have
	\[
	f_1\varepsilon_1 = \varepsilon_2, \quad 
	f_2\varepsilon_2 = \varepsilon_{\overline{2}}, \quad 
	f_1\varepsilon_{\overline{2}} = \varepsilon_{\overline{1}}.
	\]
	
	We now consider the finite dimensional irreducible representation $L(\lambda)$ of $\mathfrak{sp}_{4}$ with highest weight $\lambda = (m_1 + m_2)\epsilon_1 + m_2\epsilon_2$. We can realize $L(\lambda)$ as follows. Consider the representation
	\[ W \coloneqq  V^{\otimes m_1} \otimes (\wedge^{2} V)^{\otimes m_2}. \]
	The element
	\[ v_\lambda \coloneqq \varepsilon_1^{\otimes m_1} \otimes (\varepsilon_1 \wedge \varepsilon_2)^{\otimes m_2} \]
	of $W$ is then a maximal vector, and the $\mathfrak{sp}_{4}$-submodule of $W$ generated by $v_\lambda$ is a model for $L(\lambda)$.
	
	By weakening the conditions in the definition of the KN tableaux of $\mathfrak{sp}_4$, we define the column-strict Young tableaux of $\mathfrak{sp}_{4}$ as follows.
	
	\begin{definition}
		A column-strict Young tableau $Y$ of $\mathfrak{sp}_{4}$ of shape $\lambda$ is a filling of the Young diagram of shape $\lambda$ with entries from $\mathcal{N} = \{1, 2, \overline{2}, \overline{1}\}$, such that the entries in $Y$ are strictly increasing down each column. Let $\mathrm{CST}_{\lambda}(4)$ denote the set of all column-strict Young tableaux of $\mathfrak{sp}_{4}$ of shape $\lambda$.
	\end{definition}
	
	Each column-strict Young tableau $Y \in \mathrm{CST}_{\lambda}(4)$ corresponds to a vector $u(Y) \in W$. The correspondence is defined as follows:
	\[
	Y = \begin{ytableau}
		j_{m_2} & \cdots & j_{1} & i_{m_1}& \cdots & i_1\\
		k_{m_2} & \cdots & k_{1}
	\end{ytableau}
	\, \longmapsto \, u(Y) = \varepsilon_{i_1} \otimes \cdots \otimes \varepsilon_{i_{m_1}} \otimes
	(\varepsilon_{j_1} \wedge \varepsilon_{k_1}) \otimes \cdots \otimes (\varepsilon_{j_{m_2}} \wedge \varepsilon_{k_{m_2}}).
	\]
	For example:
	\[
	Y_1 = \begin{ytableau}
		1 & 2 & 2 \\
		2 & \overline{2}
	\end{ytableau}
	\quad \longmapsto \quad u(Y_1) = \varepsilon_2 \otimes (\varepsilon_2 \wedge \varepsilon_{\overline{2}}) \otimes (\varepsilon_1 \wedge \varepsilon_2)
	\]
	\[
	Y_2 = \begin{ytableau}
		1 & 1 & 2 \\
		2 & \overline{1}
	\end{ytableau}
	\quad \longmapsto \quad u(Y_2) = \varepsilon_2 \otimes (\varepsilon_1 \wedge \varepsilon_{\overline{1}}) \otimes (\varepsilon_1 \wedge \varepsilon_2)
	\]
	
	\begin{lemma}\label{linear-independence}
		Let $L=\{ u(Y) \mid Y\in \mathrm{CST}_{\lambda}(4)\}$, then the vectors in $L$ are linearly independent.
	\end{lemma}
	\begin{proof}
		Since
		\[\mathcal{A}=
		\left\{
		\varepsilon_{i_1} \otimes \cdots \otimes \varepsilon_{i_{m_1}} \otimes
		(\varepsilon_{j_1} \wedge \varepsilon_{k_1}) \otimes \cdots \otimes (\varepsilon_{j_{m_2}} \wedge \varepsilon_{k_{m_2}})
		\;\middle|\;
		\begin{aligned}
			& i_r \in \{1, 2, \overline{2}, \overline{1}\}, \\
			& j_s, k_s \in \{1, 2, \overline{2}, \overline{1}\}, \; j_s < k_s, \\
			& r = 1, \dots, m_1,\, s = 1, \dots, m_2
		\end{aligned}
		\right\}
		\]
		forms a basis of the representation space $W$, it is obvious that $L\subseteq \mathcal{A}$ (in fact, $L= \mathcal{A}$), and thus the vectors in $L$ are linearly independent.
	\end{proof}
	
	Next, we define a total order on $\mathrm{CST}_{\lambda}(4)$. For $Y \in \mathrm{CST}_{\lambda}(4)$, let $Y(i,j)$ denote the entry in row $i$ and column $j$ of $Y$. Note that the range of the column index $j$ depends on the row index $i$. We call a pair $(i,j)$ admissible if $Y(i,j)$ makes sense.
	
	\begin{definition}\label{total order}
		For admissible pairs $(i,j)$ and $(i',j')$, we say $(i,j) < (i',j')$ if either $j > j'$, or $j = j'$ and $i < i'$. For $Y, Y' \in \mathrm{CST}_{\lambda}(4)$ with $Y \neq Y'$, we say $Y < Y'$ if $Y(i,j) < Y'(i,j)$ for the smallest admissible pair $(i,j)$ such that $Y(i,j) \neq Y'(i,j)$. For any $Y, Y' \in \mathrm{CST}_{\lambda}(4)$, we say $Y \leq Y'$ if either $Y = Y'$ or $Y < Y'$.
	\end{definition}
	
	Recall the two column-strict Young tableaux $Y_1$ and $Y_2$ of $\mathfrak{sp}_{4}$ described above. According to Definition \ref{total order}, we have $Y_1 \geq Y_2$. Furthermore, the total order defined on the column-strict Young tableaux of $\mathfrak{sp}_{4}$ naturally extends to the corresponding vectors, so we have $u(Y_1) \geq u(Y_2)$.
	
	\begin{proposition}\label{largest tableau}
		For the Verma vector $\bm{f^{a}} v_{\lambda}$, we have
		\[
		\bm{f^{a}} v_{\lambda} = u(T(\bm{a})) + \sum_{Y < T(\bm{a})} q(Y) u(Y),
		\]
		where $T(\bm{a})$ is the KN tableau of $\mathfrak{sp}_4$ corresponding to the Verma vector $\bm{f^{a}} v_{\lambda}$, $Y \in \mathrm{CST}_{\lambda}(4)$, and $q(Y) \in \mathbb{Z}_{\geq 0}$.
	\end{proposition}
	\begin{proof}
		We first consider the highest weight vector $v_{\lambda}$. The KN tableau $T(\bm{0})$ of $\mathfrak{sp}_4$ corresponding to $v_{\lambda}$ is shown in Figure \ref{fig:sp4 highest weight}.
		
		\begin{figure}[H]
			\centering
			\begin{tikzpicture}[scale=0.7, every node/.style={font=\large}]
				
				\draw (0,0) rectangle ++(8,1);
				\node at (1,0.5) {1};
				\node at (2,0.5) {1};
				\node at (3,0.5) {$\cdots$};
				\node at (4,0.5) {1};
				\draw (5,0) -- (5,1);
				\node at (5.5,0.5) {1};
				\node at (6.5,0.5) {$\cdots$};
				\node at (7.5,0.5) {1};
				
				\draw (0,-1) rectangle ++(5,1);
				\node at (1,-0.5) {2};
				\node at (2,-0.5) {2};
				\node at (3,-0.5) {$\cdots$};
				\node at (4,-0.5) {2};
				
				\draw[decorate,decoration={brace,amplitude=5pt}] 
				(5,1.2) -- (8,1.2) node[midway,above=3pt] {$m_1$};
				\draw[decorate,decoration={brace,amplitude=5pt,mirror}] 
				(0,-1.5) -- (5,-1.5) node[midway,below=3pt] {$m_2$};
			\end{tikzpicture}
			\caption{The KN tableau of $\mathfrak{sp}_4$ corresponding to the highest weight vector $v_{\lambda}$}
			\label{fig:sp4 highest weight}
		\end{figure}
		
		Clearly, $T(\bm{0})$ is the smallest tableau in $\mathrm{CST}_{\lambda}(4)$. In this case, $v_\lambda = u(T(\bm{0})) = \varepsilon_1^{\otimes m_1} \otimes (\varepsilon_1 \wedge \varepsilon_2)^{\otimes m_2}$, and the claim holds.
		
		For $Y \in \mathrm{CST}_{\lambda}(4)$, let $k^{i}_{j}(Y)$ denote the number of entries equal to $j$ in row $i$ of $Y$, where $i \in \{1, 2\}$ and $j \in \mathcal{N} = \{1, 2, \overline{2}, \overline{1}\}$.
		Next, we consider $\bm{f^{a}} v_\lambda=f_1^{a_4} f_2^{a_3} f_1^{a_2} f_2^{a_1} v_\lambda$, where $\bm{a}= (a_{1}, a_{2}, a_{3}, a_{4}) \neq (0,0,0,0)$.
		
		We first compute $f_2 v_\lambda$:
		\begin{align*}
			f_2 v_\lambda &= f_2(\varepsilon_1^{\otimes m_1} \otimes (\varepsilon_1 \wedge \varepsilon_2)^{\otimes m_2}) \\
			&= \varepsilon_1^{\otimes m_1} \otimes (\varepsilon_1 \wedge \varepsilon_{\overline{2}}) \otimes (\varepsilon_1 \wedge \varepsilon_2)^{\otimes (m_2-1)} \\
			&\quad + \varepsilon_1^{\otimes m_1} \otimes (\varepsilon_1 \wedge \varepsilon_2) \otimes (\varepsilon_1 \wedge \varepsilon_{\overline{2}}) \otimes (\varepsilon_1 \wedge \varepsilon_2)^{\otimes (m_2-2)} \\
			&\quad + \cdots \\
			&\quad + \varepsilon_1^{\otimes m_1} \otimes (\varepsilon_1 \wedge \varepsilon_2)^{\otimes (m_2-1)} \otimes (\varepsilon_1 \wedge \varepsilon_{\overline{2}}).
		\end{align*}
		The above calculation can be expressed in terms of $\mathrm{CST}_{\lambda}(4)$:
		\begin{figure}[H]
			\centering
			\begin{tikzpicture}[scale=0.6, every node/.style={font=\normalsize}]
				\node at (-0.5,0) {\large$f_2$};
				\begin{scope}[xshift=0cm]
					\draw (0,0) rectangle ++(6,1);
					\node at (0.5,0.5) {1};
					\node at (1.25,0.5) {1};
					\node at (2,0.5) {$\cdots$};
					\draw (4,0) -- (4,1);
					\node at (2.75,0.5) {1};
					\node at (3.5,0.5) {1};
					\node at (4.25,0.5) {1};
					\node at (5,0.5) {$\cdots$};
					\node at (5.75,0.5) {1};
					\draw (0,-1) rectangle ++(4,1);
					\node at (0.5,-0.5) {2};
					\node at (1.25,-0.5) {2};
					\node at (2,-0.5) {$\cdots$};
					\node at (2.75,-0.5) {2};
					\node at (3.5,-0.5) {2};
				\end{scope}
				
				\node at (6.5,0) {$=$};
				
				\begin{scope}[xshift=7cm]
					\begin{scope}[xshift=0cm]
						\draw (0,0) rectangle ++(6,1);
						\node at (0.5,0.5) {1};
						\node at (1.25,0.5) {1};
						\node at (2,0.5) {$\cdots$};
						\draw (4,0) -- (4,1);
						\node at (2.75,0.5) {1};
						\node at (3.5,0.5) {1};
						\node at (4.25,0.5) {1};
						\node at (5,0.5) {$\cdots$};
						\node at (5.75,0.5) {1};
						\draw (0,-1) rectangle ++(4,1);
						\node at (0.5,-0.5) {2};
						\node at (1.25,-0.5) {2};
						\node at (2,-0.5) {$\cdots$};
						\node at (2.75,-0.5) {2};
						\node at (3.5,-0.5) {$\overline{2}$};
					\end{scope}
					
					\node at (6.5,0) {+};
					\begin{scope}[xshift=7cm]
						\draw (0,0) rectangle ++(6,1);
						\node at (0.5,0.5) {1};
						\node at (1.25,0.5) {1};
						\node at (2,0.5) {$\cdots$};
						\draw (4,0) -- (4,1);
						\node at (2.75,0.5) {1};
						\node at (3.5,0.5) {1};
						\node at (4.25,0.5) {1};
						\node at (5,0.5) {$\cdots$};
						\node at (5.75,0.5) {1};
						\draw (0,-1) rectangle ++(4,1);
						\node at (0.5,-0.5) {2};
						\node at (1.25,-0.5) {$\cdots$};
						\node at (2,-0.5) {2};
						\node at (2.75,-0.5) {$\overline{2}$};
						\node at (3.5,-0.5) {2};
					\end{scope}
					
					\node at (1,-3) {+\;$\cdots$\;+};
					\begin{scope}[xshift=2.5cm, yshift=-3cm]
						\draw (0,0) rectangle ++(6,1);
						\node at (0.5,0.5) {1};
						\node at (1.25,0.5) {1};
						\node at (2,0.5) {$\cdots$};
						\draw (4,0) -- (4,1);
						\node at (2.75,0.5) {1};
						\node at (3.5,0.5) {1};
						\node at (4.25,0.5) {1};
						\node at (5,0.5) {$\cdots$};
						\node at (5.75,0.5) {1};
						\draw (0,-1) rectangle ++(4,1);
						\node at (0.5,-0.5) {$\overline{2}$};
						\node at (1.25,-0.5) {2};
						\node at (2,-0.5) {$\cdots$};
						\node at (2.75,-0.5) {2};
						\node at (3.5,-0.5) {2};
					\end{scope}
				\end{scope}
				
			\end{tikzpicture}.
		\end{figure}
		Clearly, the largest tableau is
		\[ 
		\begin{tikzpicture}[scale=0.7, baseline=-0.5ex, every node/.style={font=\normalsize}]
			\begin{scope}[xshift=0cm]
				\draw (0,0) rectangle ++(6,1);
				\node at (0.5,0.5) {1};
				\node at (1.25,0.5) {1};
				\node at (2,0.5) {$\cdots$};
				\draw (4,0) -- (4,1);
				\node at (2.75,0.5) {1};
				\node at (3.5,0.5) {1};
				\node at (4.25,0.5) {1};
				\node at (5,0.5) {$\cdots$};
				\node at (5.75,0.5) {1};
				\draw (0,-1) rectangle ++(4,1);
				\node at (0.5,-0.5) {2};
				\node at (1.25,-0.5) {2};
				\node at (2,-0.5) {$\cdots$};
				\node at (2.75,-0.5) {2};
				\node at (3.5,-0.5) {$\overline{2}$};
			\end{scope}
		\end{tikzpicture}\;,
		\]
		which is precisely the KN tableau of $\mathfrak{sp}_4$ corresponding to $f_2 v_\lambda$.
		
		Next, we compute $f_2^{a_1} v_\lambda$:
		\[
		f_2^{a_1} v_\lambda = a_1! \sum_{1 \le i_1 < i_2 < \dots < i_{a_1} \le m_2} \varepsilon_1^{\otimes m_1} \otimes \biggl( \bigotimes_{t=1}^{m_2} w_t^{(i_1,\dots,i_{a_1})} \biggr),
		\]
		where
		\[
		w_t^{(i_1,\dots,i_{a_1})} = \begin{cases}
			\varepsilon_1 \wedge \varepsilon_{\overline{2}}, & t \in \{i_1, i_2, \dots, i_{a_1}\}, \\
			\varepsilon_1 \wedge \varepsilon_2, & t \notin \{i_1, i_2, \dots, i_{a_1}\}.
		\end{cases}
		\]
		Expressing this result in terms of $\mathrm{CST}_{\lambda}(4)$, we obtain
		\begin{align}\label{f_2^{a_1}}
			f_2^{a_1}T(\bm{0})=a_1! \sum_{U} U,
		\end{align}
		where the sum runs over all $U \in \mathrm{CST}_{\lambda}(4)$ satisfying $k^{1}_{1}(U)=m_1+m_2, k^{2}_{\overline{2}}(U)=a_1$, and $k^{2}_{2}(U)=m_2-a_1$. Let $T_1$ be the largest tableau among all the summands. Then $T_1$ is given by
		\[
		\begin{tikzpicture}[scale=0.7, baseline=-0.5ex, every node/.style={font=\normalsize}]
			\begin{scope}[xshift=0cm]
				
				\draw (0,0) rectangle ++(7,1);
				\node at (0.5,0.5) {1};
				\node at (1.25,0.5) {$\cdots$};
				\node at (2,0.5) {1};
				\node at (2.75,0.5) {1};
				\node at (3.5,0.5) {$\cdots$};
				\node at (4.25,0.5) {1};
				\node at (5,0.5) {1};
				\node at (5.75,0.5) {$\cdots$};
				\node at (6.5,0.5) {1};
				
				\draw (0,-1) rectangle (4.75,0); 
				\node at (0.5,-0.5) {2};
				\node at (1.25,-0.5) {$\cdots$};
				\node at (2,-0.5) {2};
				\node at (2.75,-0.5) {$\overline{2}$};
				\node at (3.5,-0.5) {$\cdots$};
				\node at (4.25,-0.5) {$\overline{2}$};
				
				\draw[decorate,decoration={brace,mirror,amplitude=4pt}]
				(2.5,-1.1) -- (4.75,-1.1) node[midway,below=2pt] {$a_1$};
			\end{scope}
		\end{tikzpicture}\;,
		\]
		Clearly, $T_1$ is the KN tableau of $\mathfrak{sp}_4$ corresponding to $f_2^{a_1}v_\lambda$. By \eqref{f_2^{a_1}}, we have
		\[ f_2^{a_1}T(\bm{0})=a_1! \,T_1+a_1! \sum_{U<T_1} U. \]
		
		Consider $f_1^{a_2}f_2^{a_1}T(\bm{0})$:
		\begin{align}\label{f_1^{a_2}f_2^{a_1}}
			f_1^{a_2}f_2^{a_1}T(\bm{0})=a_1! \,\left( f_1^{a_2}T_1\right) +a_1! \sum_{U<T_1} \left( f_1^{a_2}U\right) .
		\end{align}
		Let $T_2$ be the largest tableau in the expansion of $f_1^{a_2}f_2^{a_1}T(\bm{0})$. Then $T_2$ appears only in the expansion of $f_1^{a_2}T_1$ and does not appear in any $f_1^{a_2}U$. Furthermore, the coefficient of $T_2$ in $f_1^{a_2}T_1$ is $a_2!$. It is straightforward to see that $T_2$ is the KN tableau of $\mathfrak{sp}_4$ corresponding to $f_1^{a_2}f_2^{a_1}v_\lambda$. Write
		\[ f_1^{a_2}T_1=a_2!\,T_2+\sum_{X'<T_2}h'(X')X', \]
		where $h'(X')\in \mathbb{Z}_{\geq 0}$, and each $X'$ is a tableau in $\mathrm{CST}_{\lambda}(4)$ strictly smaller than $T_2$. Substituting the expression for $f_1^{a_2}T_1$ into \eqref{f_1^{a_2}f_2^{a_1}} yields
		\begin{align*}
			f_1^{a_2}f_2^{a_1}T(\bm{0})&=a_1! \,\biggl( a_2!\,T_2+\sum_{X'<T_2}h'(X')X'\biggr) +a_1! \sum_{U<T_1} \left( f_1^{a_2}U\right) \\
			&=h_2T_2+\sum_{X<T_2}h(X)X , 
		\end{align*}
		where $h_2=a_1! a_2!$, $h(X)\in \mathbb{Z}_{\geq 0}$, and each $X$ is a tableau in $\mathrm{CST}_{\lambda}(4)$ strictly smaller than $T_2$. 
		
		Consider $f_2^{a_3}f_1^{a_2}f_2^{a_1}T(\bm{0})$:
		\begin{align}\label{f_2^{a_3}f_1^{a_2}f_2^{a_1}}
			f_2^{a_3}f_1^{a_2}f_2^{a_1}T(\bm{0})=h_2\left( f_2^{a_3}T_2\right) +\sum_{X<T_2}h(X)\left( f_2^{a_3}X\right) .
		\end{align}
		Let $T_3$ be the largest tableau in the expansion of $f_2^{a_3}f_1^{a_2}f_2^{a_1}T(\bm{0})$. Then $T_3$ appears only in the expansion of $f_2^{a_3}T_2$ and does not appear in any $f_2^{a_3}X$. Furthermore, the coefficient of $T_3$ in $f_2^{a_3}T_2$ is $a_3!$. It is straightforward to see that $T_3$ is the KN tableau of $\mathfrak{sp}_4$ corresponding to $f_2^{a_3}f_1^{a_2}f_2^{a_1}v_\lambda$. Write
		\[ f_2^{a_3}T_2=a_3!\, T_3+\sum_{Z'<T_3}p'(Z')Z', \]
		where $p'(Z')\in \mathbb{Z}_{\geq 0}$, and each $Z'$ is a tableau in $\mathrm{CST}_{\lambda}(4)$ strictly smaller than $T_3$. Substituting the expression for $f_2^{a_3}T_2$ into \eqref{f_2^{a_3}f_1^{a_2}f_2^{a_1}} yields
		\begin{align*}
			f_2^{a_3}f_1^{a_2}f_2^{a_1}T(\bm{0})&=h_2\biggl( a_3!\, T_3+\sum_{Z'<T_3}p'(Z')Z'\biggr) +\sum_{X<T_2}h(X)\left( f_2^{a_3}X\right)\\
			&=h_3T_3+\sum_{Z<T_3}p(Z)Z , 
		\end{align*}
		where $h_3=h_2 a_3!$, $p(Z)\in \mathbb{Z}_{\geq 0}$, and each $Z$ is a tableau in $\mathrm{CST}_{\lambda}(4)$ strictly smaller than $T_3$.
		
		Finally, consider $f_1^{a_4}f_2^{a_3}f_1^{a_2}f_2^{a_1}T(\bm{0})$:
		\begin{align}\label{a}
			\bm{f^{a}}T(\bm{0})=h_3\left( f_1^{a_4}T_3\right) +\sum_{Z<T_3}p(Z)\left( f_1^{a_4}Z\right).
		\end{align}
		Let $T_4$ be the largest tableau in the expansion of $\bm{f^{a}}T(\bm{0})$. Then $T_4$ appears only in the expansion of $f_1^{a_4}T_3$ and does not appear in any $f_1^{a_4}Z$. Furthermore, the coefficient of $T_4$ in $f_1^{a_4}T_3$ is $a_4!$. It is straightforward to see that $T_4$ is exactly $T(\bm{a})$, the KN tableau of $\mathfrak{sp}_4$ corresponding to $\bm{f^{a}}v_\lambda$. Write
		\[ f_1^{a_4}T_3=a_4!\,T(\bm{a})+\sum_{Y'<T(\bm{a})}q'(Y')Y', \]
		where $q'(Y')\in \mathbb{Z}_{\geq 0}$, and each $Y'$ is a tableau in $\mathrm{CST}_{\lambda}(4)$ strictly smaller than $T(\bm{a})$. Substituting the expression for $f_1^{a_4}T_3$ into \eqref{a} yields
		\begin{align*}
			\bm{f^{a}}T(\bm{0})&=h_3\biggl( a_4!\,T(\bm{a})+\sum_{Y'<T(\bm{a})}q'(Y')Y'\biggr) +\sum_{Z<T_3}p(Z)\left( f_1^{a_4}Z\right) \\
			&=h_3a_4!\,T(\bm{a})+\sum_{Y<T(\bm{a})}q''(Y)Y \\
			&=T(\bm{a})+\sum_{Y\leq T(\bm{a})}q(Y)Y, 
		\end{align*}
		where $q(Y), q''(Y)\in \mathbb{Z}_{\geq 0}$, and each $Y$ is a tableau in $\mathrm{CST}_{\lambda}(4)$. Therefore, we have
		\[ \bm{f^{a}}v_\lambda=u(T(\bm{a}))+\sum_{Y\leq T(\bm{a})}q(Y)u(Y). \]
		
		Thus, the proposition holds.
	\end{proof}
	
	\begin{theorem}
		All Verma vectors of the irreducible representation $L(\lambda)$ of $\mathfrak{sp}_{4}$ are linearly independent. 
	\end{theorem}
	\begin{proof}
		Let the set of KN tableaux of $\mathfrak{sp}_4$ be given by
		\[ \mathrm{KN}_{\lambda}(4)=\left\lbrace T(\bm{a_1})> T(\bm{a_2})> \cdots> T(\bm{a_l}) \right\rbrace. \]
		Then the corresponding set of Verma vectors is
		\[ F=\left\lbrace \bm{f^{a_1}} v_\lambda, \bm{f^{a_2}} v_\lambda, \cdots, \bm{f^{a_l}} v_\lambda \right\rbrace. \]
		
		Suppose there exist $x_1, x_2, \ldots, x_l \in \mathbb{C}$ such that
		\[
		x_1 \bm{f^{a_1}} v_\lambda + x_2 \bm{f^{a_2}} v_\lambda + \dots + x_l \bm{f^{a_l}} v_\lambda = 0.
		\]
		Then by Proposition~\ref{largest tableau}, we have
		\begin{align*}
			&x_1\biggl( u(T(\bm{a_1})) + \sum_{Y_1 \leq T(\bm{a_1})} q(Y_1)\, u(Y_1)\biggr)
			+ x_2\biggl( u(T(\bm{a_2})) + \sum_{Y_2 \leq T(\bm{a_2})} q(Y_2)\, u(Y_2)\biggr) \\
			& + \cdots + x_l\biggl( u(T(\bm{a_l})) + \sum_{Y_l \leq T(\bm{a_l})} q(Y_l)\, u(Y_l)\biggr) = 0.
		\end{align*}
		Since $T(\bm{a_1}) > T(\bm{a_2}) > \cdots > T(\bm{a_l})$, it follows that
		\[
		u(T(\bm{a_1})) > u(T(\bm{a_2})) > \cdots > u(T(\bm{a_l})).
		\]
		Combined with Lemma~\ref{linear-independence}, we deduce that $x_1 = 0$. Now assume that for some $i \in \left\lbrace 2, 3, \ldots, l \right\rbrace$, we have $x_1 = x_2 = \cdots = x_{i-1} = 0$. Then
		\begin{align*}
			&x_i\biggl( u(T(\bm{a_i})) + \sum_{Y_i \leq T(\bm{a_i})} q(Y_i)\, u(Y_i)\biggr)
			+ x_{i+1}\biggl( u(T(\bm{a_{i+1}})) + \sum_{Y_{i+1} \leq T(\bm{a_{i+1}})} q(Y_{i+1})\, u(Y_{i+1})\biggr) \\
			&\quad + \cdots 
			+ x_l\biggl( u(T(\bm{a_l})) + \sum_{Y_l \leq T(\bm{a_l})} q(Y_l)\, u(Y_l)\biggr) = 0.
		\end{align*}
		Since $u(T(\bm{a_i})) > u(T(\bm{a_{i+1}})) > \cdots > u(T(\bm{a_l}))$, we obtain $x_i = 0$. By mathematical induction, we have
		\[
		x_1 = x_2 = \cdots = x_l = 0.
		\]
		
		Therefore, all the Verma vectors are linearly independent. 
	\end{proof}
	
	By Theorem \ref{KN-correspondence}, there is a one-to-one correspondence between the set $F$ of Verma vectors of $L(\lambda)$ and the set $\mathrm{KN}_{\lambda}(4)$ of KN tableaux of $\mathfrak{sp}_4$, hence
	\[
	|F| = |\mathrm{KN}_{\lambda}(4)|.
	\]
	Combined with Proposition \ref{prop:dimension}, we obtain
	\[
	\dim L(\lambda) = |F|.
	\]
	Since the Verma vectors in $F$ are linearly independent, $F$ forms the Verma basis of $L(\lambda)$. We summarize this result in the following theorem.
	
	\begin{theorem}\label{Verma basis}
		The set $F$ of Verma vectors is the Verma basis of the finite dimensional irreducible representation $L(\lambda)$ of $\mathfrak{sp}_4$.
	\end{theorem}
	
	\begin{example}
		All the Verma vectors in Example \textnormal{\ref{example}} form the Verma basis of $L(\lambda)$ with the highest weight $\lambda = 3\epsilon_1 + 2\epsilon_2$.
	\end{example}
	
	
	\section*{Acknowledgments}
The authors are partially supported by NSFC (Grant No. 12161090).

\end{document}